\documentclass[a4paper,12pt]{amsart}
\usepackage{amsmath}
\usepackage{amssymb}
\usepackage{mathrsfs}
\usepackage{enumerate}
\usepackage{ifthen}
\usepackage{graphicx}
\usepackage[T1]{fontenc} 

\nonstopmode \numberwithin{equation}{section}
\newtheorem{thm}{Theorem}[section]

\theoremstyle{definition}

\newcounter{minutes}
\divide\time by 60
\newcounter{hours}
\multiply\time by 60
\addtocounter{minutes}{-\time}

\newcounter {own}
\def\theown {\thesection       .\arabic{own}}

{\qed\bigskip}

\newcounter{alphabet}

\begin{document}
	
	\title{Pre-Schwarzian norm estimate and characterization of certain harmonic mappings}

	\author{Sushil Pandit}
	\address{Sushil Pandit,
		Department of Mathematics,
		National Institute of Technology Meghalaya,
		Meghalaya-793108, India, Orcid Id-0000-0002-4129-6456}
	\email{sushilpandit15594@gmail.com}

	
	\def\thefootnote{}
	\footnotetext{ {\tiny File:~\jobname.tex,
			printed: \number\year-\number\month-\number\day,
			\thehours.\ifnum\theminutes<10{0}\fi\theminutes }
	} \makeatletter\def\thefootnote{\@arabic\c@footnote}\makeatother

	\begin{abstract}
		In this article, we consider certain class of harmonic mappings defined in the unit disk $\mathbb{D}=\{z\in\mathbb{C}: |z|<1\}.$ Then we obtain pre-Schwarzian norm estimate of functions in the class. Next, we show that functions in the considered class are univalent and close-to-convex. Moreover, we discuss some growth and distortion theorems for associated analytic and co-analytic parts of harmonic mappings in the class. At last, we present coefficient estimate for the analytic part.
	\end{abstract}

	\thanks{}
	
	\maketitle
	\pagestyle{myheadings}
	\markboth{Sushil Pandit}{Pre-Schwarzian norm estimate and characterization of certain harmonic mappings}
	
	\section{Introduction}

	Let $\mathcal{A}_0$ denote the class of all analytic functions $h$ in the unit disk $\mathbb{D}=\{z\in\mathbb{C}:|z|<1\}.$
	Let $\mathcal{A}$ be the class of functions $h\in\mathcal{A}_0$ with the normalization $h(0)=h'(0)-1=0.$ Let $\mathcal{S}$ be the class of functions $h\in\mathcal{A}$ that are univalent. A function $h\in\mathcal{S}$ is called convex in $\mathbb{D}$ if the image domain $h(\mathbb{D})$ is convex. A domain $\Omega$ is called convex if line segment joining any two points in $\Omega$ lies in $\Omega.$ The set of all convex functions in $\mathcal{S}$ is denoted by $\mathcal{C}$. It is well-known that (see \cite{Duren-1983}) a function $h\in\mathcal{S}$ is in $\mathcal{C}$ if and only if
	\begin{align*}
		{\rm Re\,}\left(1+z\frac{h''(z)}{h'(z)}\right)>0
	\end{align*}
	for $z\in\mathbb{D}.$ It is clear from the above characterization that the function $h(z)=z/(1-z)$ belongs to the class $\mathcal{C}.$ Indeed, the function $h$ maps $\mathbb{D}$ univalently onto ${\rm Re\,}w >-1/2$.
	
 A continuous twice differentiable complex valued function $f=u+iv$ is called harmonic in a domain $\Omega$ if both $u$ and $v$ are real harmonic in $\Omega.$ In any simply connected domain $\Omega,$ every harmonic mapping $f$ can be represented as $f=h+\overline{g},$ where $h$ and $g$ are analytic in $\Omega.$ This representation is known as canonical representation and $h$ is called analytic part whereas $g$ is called co-analytic part of $f.$ A harmonic mapping $f=h+\overline{g}$ is sense preserving if the Jacobian $J_f=|h'|^2-|g'|^2$ is positive and sense reversing if $J_f$ is negative. Lewy \cite{Lewy-1936} showed that $f$ is locally univalent if $J_f$ is non vanishing. Let $\mathcal{H}$ be class of all sense preserving harmonic functions $f=h+\overline{g}$ in $\mathbb{D}$ of the form
    \begin{align}\label{P-151}
		h(z)=z+\sum\limits_{n=2}^\infty a_nz^n,~g(z)=\sum\limits_{n=1}^\infty b_nz^n.
	\end{align}
	Let $\mathcal{S_H}$ denote the family of functions in $\mathcal{H}$ that are univalent in $\mathbb{D}$. In 1984, Clunie and Sheil-Small \cite{Clunie-Small-1984} studied several geometric properties of functions in $\mathcal{S_H}$ alongwith functions in its  subclasses of starlike functions, convex functions, close-to-convex functions.

	\subsection{Pre-Schwarzian and Schwarzian norms:}
	For a locally univalent analytic function $h$, the quantity 
	\begin{align}\label{P-141}
	||P_h|| = \sup_{z \in \mathbb{D}}(1-|z|^2)|P_h(z)|
	\end{align}
	is known as pre-Schwarzian norm of $h.$ Here, $P_h=\frac{h''}{h'}$ is called pre-Schwarzian derivative of $h.$ For a univalent function $h$, it is well known that $||P_h||\leq 6$ (see \cite{Kraus-1932}) and the estimate is sharp. On the other hand, for a locally univalent function $h$ in $\mathcal{A}$, it is known that if $||P_h||\leq 1$ (see \cite{Becker-1972}, \cite{Becker-Pommerenke-1984}), then the function $h$ is univalent in $\mathbb{D}$. In 1976, Yamashita \cite{Yamashita-1976} proved that $||P_h||$ is finite if and only if $h$ is uniformly locally univalent in $\mathbb{D}.$ In recent years, Firoz Ali and Sanjit Pal \cite{Ali-Pal-2023, Ali-Pal--2023, Ali-Pal---2023, Ali-Pal-2024} have studied pre-Schwarzian norm for different classes of analytic functions. 
	
	Several important global univalence criteria for a locally univalent analytic function $h$ were obtained using the notions of pre-Schwarzian derivative of $h$. These applications of the derivative generated a natural question if the concept can be generalized to harmonic mappings. Affirmatively, Kanas and Klimek-Sm\c{e}t \cite{Kanas-Smet-2014} proposed a definition of pre-Schwarzian derivative (which we denote by $\mathbf{P}_f$) for a locally univalent and sense preserving harmonic mapping $f = h+\overline{g}$ with its dilatation $\omega = g'/h'$ of the form $\omega = p^2$ for some analytic function $p$. The definition is given as follows:
	\begin{align}\label{p1-020}
		\mathbf{P}_f= \frac{2\partial(\log\lambda)}{\partial z}= \frac{h''}{h'}+\frac{2\overline{p}p'}{1+|p|^2},~~\text{where}~\lambda = |h'|+|g'|.
	\end{align}
	The pre-Schwarzian norm is defined similar to \eqref{P-141}. Ali and Pandit in \cite{Ali-Pandit-2023} have obtained estimate of $\mathbf{P}_f$ for different classes of close-to-convex harmonic mappings having different analytic parts. 
	
	The pre-Schwarzian derivative given in  \eqref{p1-020}  requires the dilatation to meet the condition $\omega = p^2$ that does not always hold for  $f=h+\overline{g}$. Due to this, in many cases one can not define it globally  for an univalent harmonic mapping. In 2015, Hern{\'a}ndez and Mart{\'\i}n \cite{Hernandez-Martin-2015} defined the pre-Schwarzian derivative of a locally univalent harmonic mapping $f=h+\overline{g}$ by\\
	\begin{align}\label{p1-027}
		P_f = \left(\log J_f\right)_z = \frac{h''}{h'}-\frac{\overline{\omega}\omega'}{1-|\omega|^2},
	\end{align}
	where $J_f$ is the Jacobian of $f$. Again, the pre-Schwarzian norm is defined same as \eqref{P-141}. In 2015, Hern{\'a}ndez and Mart{\'\i}n \cite{Hernandez-Martin-2015} proved the sharp estimate $\|P_f\|\le 5$ for convex harmonic mappings $f$. In 2016, Graf \cite{GRAF-2016} obtained sharp estimate of $\|P_f\|$ for any locally univalent harmonic mapping $f$ in a affine and linear invariant family. In 2019, Liu and Ponnusamy \cite{Liu-Ponnusamy-2018} obtained the sharp estimates of the pre-Schwarzian norm $\|P_f\|$ for stable harmonic univalent functions and stable harmonic convex functions. Ali and Pandit \cite{Ali-Pandit-2023} provided a different method of constructing extremal function for sharpness of pre-Schwarzian norm estimates. Moreover, in the same article \cite{Ali-Pandit-2023}, authors considered a class of close-to-convex harmonic mappings and obtained estimate of $\|P_f\|.$ In this article, we obtain sharp estimate of pre-Schwarzian norm of certain harmonic mappings.\\

		\subsection{Bloch Functions:}
	
	There are many known results concerning the order of growth of the first and higher derivatives of bounded univalent functions in the unit circle. An interesting outcome of these results for functions $h$ analytic in $\mathbb{D}$ is $|h'(z)| = O((1-|z|)^{-3})$ when $h$ is univalent and $|h^{n}(z )|=O((1-|z|)^{-n})$ when $h$ is bounded. All these investigations fail to give a proper analysis of the behavior of the quantity $|h'(z)|(1-|z|)$ as $|z|\rightarrow 1$ from the interior of the unit circle (see \cite{Seidel-Walsh-1942}). Study of this problem is one of the main reason of discovery of the concept of Bloch functions. An analytic function $h$  defined in $\mathbb{D}$ is called Bloch if $\beta_h=\sup_{z\in\mathbb{D}}(1-|z|^2)|h'(z)|<\infty.$	The collection $\mathcal{B}$ of analytic Bloch functions in $\mathbb{D}$ form a Banach space with the norm given by
	\begin{align*}
		||h||_{\mathcal{B}}=|h(0)|+\sup_{z\in\mathbb{D}}(1-|z|^2)|h'(z)|.
	\end{align*} 
	For more properties of analytic Bloch functions, we refer to the articles \cite{Pommerenke-1970, Anderson-Clunie-Pommerenke-1974,  Anderson-Shields-1976, Bonk-1991, Bonk-Minda-Yanagihara-1996}. A harmonic mapping $f=h+\overline{g}\in\mathcal{H}$ is called Bloch if and only if
	\begin{align*}
		\beta_f=\sup_{z\in\mathbb{D}}(1-|z|^2)(|h'(z)|+|g'(z)|)<\infty.
	\end{align*}
	For more information about harmonic Bloch mappings, we refer to \cite{Chen-Gauthier-Hengartner-2000, Colonna-1989}.

	\section{Main Results}\label{Sect-2}
	Analytic parts of harmonic mappings play an vital role to shape their geometric properties. In particular, if $f=h+\overline{g}\in\mathcal{H}$ is a sense preserving harmonic mapping and, $h$ is univalent and convex, then $f=h+\overline{g}$ is univalent and close-to-convex (see \cite{Clunie-Small-1984}). In 2011, Bshouty and Lyzzaik \cite{Bshouty-Lyzzaik-2011} proved that if $f=h+\overline{g}\in\mathcal{H}$ is a harmonic mapping with dilatation $\omega=z$ and the analytic part $h \in \mathcal{A}$ satisfies $${\rm Re\,}\left(1+\frac{zh''(z)}{h'(z)}\right)>-\frac{1}{2}$$ for $z\in\mathbb{D}$, then $f$ is univalent and close-to-convex. In this connection, we consider a class $\mathcal{H_R}$ of sense preserving harmonic mappings $f=h+\overline{g}\in\mathcal{H}$ in the unit disk $\mathbb{D}$ having dilatation $\omega=g'/h'$ and $h\in\mathcal{R}$ where $\mathcal{R}$ is the class of analytic functions $h\in\mathcal{A}$ in $\mathbb{D}$ such that 
	\begin{align*}
		h'(z)=\frac{m}{m-z^m},~~m\in\mathbb{N}.
	\end{align*}
In Theorem \ref{P-171}, we show that $f=h+\overline{g}\in\mathcal{H_R}$ is univalent and close-to-convex by showing that $h$ is univalent and convex. 

Our first result provides pre-Schwarzian norm estimate of a harmonic mapping $f$ in the class $\mathcal{H_R}.$
	\begin{thm}
		Let $f=h+\overline{g}\in\mathcal{H_R}\in\mathcal{H}$ be a harmonic mapping of the form \eqref{P-151}. Then pre-Schwarzian norm $\|P_f\|\le 3$. The estimate is sharp. 
	\end{thm}
	\begin{proof}
		The pre-Schwarzian norm $\|P_f\|$ of a harmonic mapping $f=h+\overline{g}\in\mathcal{H_R}$ is 
		\begin{align}\label{HM-100}
			\|P_f\|= & \sup_{z \in \mathbb{D}}\left|\frac{h''}{h'}-\frac{\overline{\omega}\omega'}{1-|\omega|^2}\right|(1-|z|^2)\\\nonumber
			\le & \sup_{z \in \mathbb{D}}\left|\frac{h''}{h'}\right|(1-|z|^2)+\sup_{z \in \mathbb{D}}\left|\frac{\overline{\omega}\omega'}{1-|\omega|^2}\right|(1-|z|^2)\\\nonumber
			\le & 	\|P_h\|+1.
		\end{align}

		By the hypothesis, it follows that 
		$h'(z)=\frac{m}{m-z^m},~m\in\mathbb{N}$
		and so the pre-Scwarzian derivative of $h$ is
		\begin{align*}
			P_h(z)=\frac{mz^{m-1}}{m-z^m}. 
		\end{align*}
		Now the pre-Schwarzian norm $\|P_h\|$ of $h,$ using Schwarz pick lemma, is 
		\begin{align*}
			\|P_h\|=&\sup\limits_{z\in\mathbb{D}}|P_h(z)|(1-|z|^2)\\
			=&\sup\limits_{z\in\mathbb{D}}\left|\frac{z^{m-1}}{1-z^m/m}\right|(1-|z|^2)\\
			\le &\sup\limits_{z\in\mathbb{D}}\frac{|z^{m-1}|}{1-\left|z^m/m\right|}(1-|z|^2)\\
			=&\sup\limits_{z\in\mathbb{D}}\frac{|\left(z^m/m\right)'|\left(1+\left|z^m/m\right|\right)}{1-\left|z^m/m\right|^2}(1-|z|^2)\\
			\le&\sup\limits_{z\in\mathbb{D}}\frac{1+\left|z^m/m\right|}{(1-|z|^2)}(1-|z|^2)\\
			\le& \sup\limits_{z\in\mathbb{D}}(1+|z|)\\
			=&2.
		\end{align*}
		Thus from \eqref{HM-100}, the pre-Schwarzian norm $\|P_f\|$ is
		$$\|P_f\|\le 3.$$

		To show that the estimate is sharp, we consider a harmonic mapping $f_q=h+\overline{g}$ with $h(z)=-\log{(1-z)}$ and dilatation $\omega_q(z)=\frac{z-q}{1-qz},~q\in(0,1).$ It is clear that $f_q\in \mathcal{H_R}$ and so $\|P_{f_q}\|\le3.$ A few calculations show that 
		\begin{align}
			P_h(z)=\frac{1}{1-z}\quad\text{and}\quad \frac{\overline{\omega_q(z)}\omega_q'(z)}{1-|\omega_q(z)|^2}=\frac{\overline{z}-q}{(1-qz)(1-|z|^2)}
		\end{align}
		from which it follows that 
		\begin{align}
			\|P_{f_q}\|=&\sup_{z \in \mathbb{D}}\left|\frac{h''(z)}{h'(z)}-\frac{\overline{\omega_q(z)}\omega_q'(z)}{1-|\omega_q(z)|^2}\right|(1-|z|^2)\\\nonumber
			=&\sup_{z \in \mathbb{D}}\left|\frac{1}{1-z}-\frac{\overline{z}-q}{(1-qz)(1-|z|^2)}\right|(1-|z|^2)\\\nonumber
			\ge & S_q
		\end{align}
		where 
		\begin{align}
			S_q=&\sup_{z \in [0,1)}\left|\frac{1}{1-z}-\frac{\overline{z}-q}{(1-qz)(1-|z|^2)}\right|(1-|z|^2)\\\nonumber\\\nonumber
			=&\sup_{r \in [0,1)}\left|1+r-\frac{r-q}{1-qr}\right|\\\nonumber
			=&\sup_{r \in [0,1)}\phi(r)\\\nonumber
		\end{align}
		with $\phi(r)=1+r-\frac{r-q}{1-qr}.$ We note that $\phi(0)=1+q,~\phi(1^-)=1.$ Next $\phi'(r)=0$ implies that $$r=\frac{1\pm\sqrt{1-q^2}}{q}.$$ We consider $\frac{1-\sqrt{1-q^2}}{q}=r_0$ as it lies in $[0,1).$ Since $r_0$ is stationary point and $$\phi(r_0)=\frac{2+q-2\sqrt{1-q^2}}{q}\rightarrow 3$$ as $q$ tends to $1,$ it follows that $S_q=3.$ Thus we get the relation $$3=S_q\le \|P_{f_q}\|\le 3$$ which implies that $\|P_{f_q}\|=3.$
	\end{proof}
Here, we note that, for a harmonic mapping $f=h+\overline{g},$ the dilatation of the form  $\omega(z)=g'(z)/h'(z)=\frac{z-q}{1-qz},~q\in(0,1)$  is used widely to get more interesting results.
	\begin{thm}\label{P-171}
		Every $f=h+\overline{g}\in\mathcal{H_R}$ is univalent and close-to-convex in $\mathbb{D}.$
	\end{thm}
	\begin{proof}
		To show that $f=h+\overline{g}\in\mathcal{H_R}$ is univalent and close-to-convex, it is sufficient to show the analytic part $h$ is univalent and convex.
		It is easy to show that ${\rm Re\,}h'(z)>0$ and so $h$ is univalent from Noshiro-Warshawski Theorem (see \cite{Duren-1983}). A simple calculation shows that 
		\begin{align}\label{v-00100}
			1+\frac{zh''(z)}{h'(z)}=\frac{1+\left(1-\frac{1}{m}\right)z^m}{1-\frac{1}{m}z^m}	
		\end{align}
		Let us consider a transformation $$z^m\rightarrow\zeta$$
		where $\zeta:\mathbb{D}\rightarrow\mathbb{D}$ a self mapping of unit disk. Next from \eqref{v-00100}, we get
		\begin{align*}
			1+\frac{zh''(z)}{h'(z)}=\frac{1-\frac{1}{m}\left(1-\frac{1}{m}\right)|\zeta(z)|^2+\left(1-\frac{1}{m}\right)\zeta(z)-\frac{1}{m}\overline{\zeta(z)}}{|1-\frac{1}{m}\zeta(z)|^2}.
		\end{align*}
		Now, for $m=1$ and $m=2$ we get 
		$${\rm Re\,}\left(1+\frac{zh''(z)}{h'(z)}\right)=\frac{1-{\rm Re\,}\left(\zeta(z)\right)}{|1-\zeta(z)|^2}\quad\text{and}\quad {\rm Re\,}\left(1+\frac{zh''(z)}{h'(z)}\right)=\frac{4-|\zeta(z)|^2}{|2-\zeta(z)|^2}$$
		respectively, which are positive for $z\in\mathbb{D}\setminus\{0\}.$ Next, for $m>2$, it follows that
		\begin{align*}
			{\rm Re\,}\left(1+\frac{zh''(z)}{h'(z)}\right)=&\frac{1-\frac{1}{m}\left(1-\frac{1}{m}\right)|\zeta(z)|^2+\left(1-\frac{2}{m}\right){\rm Re\,}\left(\zeta(z)\right)}{|1-\frac{1}{m}\zeta(z)|^2}\\
			\ge & \frac{1-\frac{1}{m}\left(1-\frac{1}{m}\right)-\left(1-\frac{2}{m}\right)}{|1-\frac{1}{m}\zeta(z)|^2}\\
			= & \frac{\frac{1}{m}+\frac{1}{m^2}}{|1-\frac{1}{m}\zeta(z)|^2}\\
			\ge & 0.
		\end{align*}
		Thus $h$ is convex in $\mathbb{D}.$ This completes the proof.
	\end{proof}
	
Next result shows that every harmonic mapping in the class $\mathcal{H_R}$ is harmonic Bloch.
	
	\begin{thm}
		Let $f=h+\overline{g}\in\mathcal{H_R}$ be a sense preserving harmonic mapping of the form \eqref{P-151}. Then $f$ is harmonic Bloch. 
	\end{thm}
	\begin{proof}
	We need to show that $\beta_f$ is finite for $f=h+\overline{g}\in\mathcal{H_R}$. By the hypothesis we get $$|h'(z)|\le\frac{1}{1-|z|}.$$ For $f\in\mathcal{H_R},$ by Schwarz Pick lemma, we have
	\begin{align*}
		\beta_f=& \sup_{z\in\mathbb{D}}(1-|z|^2)(|h'(z)|+|g'(z)|)\\
		= & \sup_{z\in\mathbb{D}}(1-|z|^2)(1+|\omega(z)|)|h'(z)|\\
		= & \sup_{z\in\mathbb{D}}(1-|z|^2)(1+|\omega(z)|)\left|\frac{m}{m-z^m}\right|\\
		\le & \sup_{z\in\mathbb{D}}(1-|z|^2)(1+|z|)\frac{1}{1-|z|}\\
		= & \sup_{z\in\mathbb{D}}(1+|z|)^2\\
		=& 4 <\infty.
	\end{align*}
	This completes the proof.
	\end{proof}
Now we discuss some growth and distortion theorems for analytic part and co-analytic part of harmonic mappings in the class $\mathcal{H_R}.$	
	\begin{thm}
	Let $f=h+\overline{g}\in\mathcal{H_R}$ be sense preserving harmonic mapping of the form \eqref{P-151}. Then, for $|z|=r,$ the sharp inequalities 
	\begin{align*}
	\frac{1}{1+r}\le|h'(z)|\le\frac{1}{1-r}\quad\text{and}\quad 0\le|g'(z)|\le\frac{1}{1-r}
	\end{align*}
	hold.
	\end{thm}
	\begin{proof}
 Since $f=h+\overline{g}\in\mathcal{H_R}$, from hypothesis, it follows that
 $$h'(z)=\frac{1}{1-z^m/m},~m\ge 1.$$ Now it is easy to see that
 \begin{align}\label{P-111}
 \frac{1}{1+r}\le|h'(z)|\le\frac{1}{1-r},~ |z|=r.
 \end{align}
 Since dilatation $\omega=g'/h'$ of $f=h+\overline{g}$ is analytic with $|\omega|<1$, so for $|\omega(0)|=t$ and $|z|=r$, we have
 \begin{align}\label{P-121}
 	\frac{r-t}{1-tr}\le |\omega(z)|\le \frac{r+t}{1+tr}.
 \end{align}
 As $g'=\omega h'$, from \eqref{P-111} and \eqref{P-121}, it follows that 
 
  \begin{align}\label{P-131}
 	\frac{r-t}{(1-tr)(1+r)}\le |g'(z)|\le \frac{r+t}{(1+tr)(1-r)}.
 \end{align}
It is simple to show that minimum value of the left hand quantity in \eqref{P-131} is $0$ at $t=r$ and maximum value of the right hand quantity in \eqref{P-131} is $\frac{1}{1-r}$ as $t$ tends to $1.$ Thus,
\begin{align*}
	0\leq |g'(z)|\leq\frac{1}{(1-r)}.
\end{align*}
 The estimates are sharp for the harmonic function $f_t=h+\overline{g}$ with $h(z)=-\log{(1-z)}$ and dilatation $\omega_t(z)=(z+t)/(1+t z),~t\in[0,1)$ because for this function, we get
 \begin{align*}
 h'(z)=\frac{1}{1-z}\quad\text{and}\quad	g'(z)=\frac{z+t}{(1+t z)(1-z)}.
 \end{align*}

	\end{proof}
	\begin{thm}\label{log-000200}
		Let $f\in\mathcal{H_R}$ be a sense-preserving harmonic mapping with dilatation $\omega.$ Then for $|z|=r$ the sharp inequalities
		\begin{align*}
			|g(z)|\le
			\begin{cases}
				-\log{(1-r)}-\frac{1-t}{t}\log{(1+tr)}  ,~~\text{when}~~|\omega(0)|=t\neq 0 \\
				 -r-\log{(1-r)},\quad\quad\quad\quad\quad\quad\quad\quad\text{when}~~\omega(0)=0
			\end{cases}
		\end{align*}
		hold.
	\end{thm}
	\begin{proof}
	It is well known that for analytic function $\omega:\mathbb{D}\rightarrow\mathbb{D}$ with $|\omega(0)|=t$, the inequality
	\begin{align}|\label{log-000300}
		\omega(z)|\le \frac{t+|z|}{1+t|z|}
	\end{align}	
	and for analytic function $h\in\mathcal{R}$, the inequality
	\begin{align}\label{log-000320}
		|h'(z)|\le \frac{1}{1-|z|}
	\end{align}	
	hold. By hypothesis it follows that $g'=\omega h'$ and so from \eqref{log-000300} and \eqref{log-000320}, for $t\ne0,$ we get
	\begin{align*}
		|g(z)|\le & \int_{0}^{r}|\omega(\zeta)||h'(\zeta)|d|\zeta|\\
		       \le & \int_{0}^{r}\frac{t+|\zeta|}{1+t|\zeta|}\frac{1}{1-|\zeta|}d|\zeta|\\
		       = & -\log{(1-r)}-\frac{1-t}{t}\log{(1+tr)}.
	\end{align*}
	On the other hand, when $\omega(0)=0,$ the inequality \eqref{log-000300} becomes $|\omega(z)|\le |z|$ and so from $g'=\omega h'$ and \eqref{log-000320}, we have
		\begin{align*}
			|g(z)|\le -r-\log{(1-r)}.
		\end{align*}
	
		To show that the estimates are sharp, we consider the harmonic mapping $f(z)=h(z)+\overline{g(z)}$ where $h(z)=-\log{(1-z)}$ and dilatation $\omega(z).$ Then the first estimate is sharp for $\omega(z)=(t+z)/(1+t z),~t\in(0,1)$ and the second estimate is sharp for $\omega(z)=z.$
	\end{proof}
	
Now we present our last result which gives some coefficient estimate for analytic part of a harmonic mapping in $\mathcal{H_R}$.	
	
	\begin{thm}
		Let $f=h+\overline{g}\in\mathcal{H_R}$ be a sense preserving harmonic mapping of the form \eqref{P-151}. Then the sharp estimate $|a_n|\le 1/n$ holds. 
	\end{thm}
	
	\begin{proof}
		By hypothesis, it follows that   $$h'(z)=\frac{m}{m-z^m}=1+\sum\limits_{k=1}^\infty \frac{z^{km}}{m^k}.$$ A simple calculation show that 
		\begin{align*}
			h(z)=\sum\limits_{k=1}^\infty \frac{z^{(k-1)m+1}}{m^{(k-1)}[(k-1)m+1]}.
		\end{align*}
		Thus the $n$th coefficient of $h$ can be given by
		\begin{align}\label{P-180}
			a_n=
			\begin{cases}
				\frac{1}{m^{(k-1)}[(k-1)m+1]} ,\quad\text{when}~~n=(k-1)m+1 \\
				0,\quad\quad\quad\quad\quad\quad\quad\text{otherwise}.
			\end{cases}
		\end{align}
Let
		\begin{align*}
		 \psi(m)=\frac{1}{m^{(k-1)}[(k-1)m+1]}
		\end{align*}
		 It is easy to see that $\psi$ is decreasing function of $m$ and so is maximum at $m=1.$ Therefore, the $n$th coefficient of $h$ is bounded above by $\psi(1)=1/k$ with $n=k$, that is,  
		 $$|a_n|\le \frac{1}{n}.$$
		 	The estimate is sharp for the harmonic mapping $f=h+\overline{g}$ with $$h(z)=-\log{(1-z)}=z+\sum\limits_{n=2}^\infty \frac{z^n}{n}.$$
		 \end{proof}


	\section*{Declarations:}

	\noindent\textbf{Data availability:}
	Data sharing not applicable to this article as no data sets were generated or analyzed during the current study.\\

	\noindent\textbf{Conflict of interest:} The author declares that he has no conflict of interest.

	
\end{document}